\documentclass[a4paper,10pt]{article} % classe du document
\usepackage[utf8]{inputenc}  %encodage d'entrée du fichier
\usepackage[english]{babel} % gère notamment l'hyphenation
\usepackage{mathtools}  % ce package charge amsmath et corrige des bugs
\usepackage{amssymb} %  utile pour plein de polices de maths (notamment mathcal)
\usepackage{mathrsfs}  % encore une police (mathscr)
\usepackage{amsthm}  % pour les environnements de maths
\usepackage{hyperref}
\usepackage{color}

\theoremstyle{plain}
\newtheorem{Thm}{Theorem}
\newtheorem*{Thm*}{Theorem} % Pour un théorème pas numéroté

\newtheorem{Lem}{Lemma}
\theoremstyle{definition}
\newtheorem{Def}{Definition}
\theoremstyle{remark}
\newtheorem{Rem}{Remark}[section]

\newtheorem{Exe}{Example}

\DeclareMathOperator*{\E}{\mathbb{E}}

\DeclareMathOperator*{\Var}{Var}
\DeclareMathOperator*{\Exp}{Exp}

\DeclareMathOperator*{\vectspan}{span}

\title{On the existence of fractional Brownian fields indexed by manifolds with closed geodesics}
\author{Nil Venet\footnote{Institut de mathématiques de Toulouse, nil.venet@math.univ-toulouse.fr}}

\begin{document} 
\maketitle

\begin{abstract}
	We give a necessary condition of geometric nature for the existence of the $H$-fractional Brownian field indexed by a Riemannian manifold. \\ In the case of the Lévy Brownian field ($H=1/2$) indexed by manifolds with minimal closed geodesics it turns out to be very strong. In particular we show that compact manifolds admitting a Lévy Brownian field are simply connected.
	
	We also derive from our result the nondegenerescence of the Lévy Brownian field indexed by hyperbolic spaces.
	
	These results stress the need for alternative kernels on nonsimply connected manifolds to allow for Gaussian modelling or kernel machine learning of functional data with manifold-valued entries.
\end{abstract}

\section{Introduction}

When it comes to handling functional data, Gaussian random processes have become an unavoidable class of models.
If a considerable literature exists in the Euclidean case, applications call for random models of functions defined on more general spaces.
However the study of Gaussian random fields indexed by metric spaces such as Riemannian manifolds or discrete spaces is still in progress,
and we do not always have satisfying models in these cases.
In particular a definition of the fractional Brownian motion indexed by any metric space was given by Istas in \cite{istas2011manifold},
but in general the existence of those fields remains an open question.

Fractional Brownian motion is a popular Gaussian random process. It is characterised by a positive parameter $H$ (the so-called Hurst exponent) that governs the roughness of its sample paths, together with a ``long-range dependence" quality.
This allows for example to fit a model on time series with long memory, or on spatial data presenting textures.
Those qualities participate in the wide success of this process in a variety of applied domains since the initial article of Mandelbrot and Van Ness (\cite{MandelbrotVanNess}), and motivate the research of an extension to non Euclidean spaces.
Let us recall that the case $H=1/2$ corresponds to the Brownian motion: it is most natural to look for generalisations of that standard Gaussian process.

Moreover, the existence of the $H$-fractional Brownian field is equivalent to the negative definiteness of the kernel $d^{2H}$, which in turn implies the positive definiteness of kernels that are crucial for ``kernel method" machine learning of nonlinear data (see for example \cite{scholkopf2002learning}).

Fractional Brownian fields do not always exist. Furthermore little is understood from the links between their existence and the structure of the index space. In \cite{istas2011manifold} Istas remarks that there exists a \emph{fractional index} $\beta_E \in [0, \infty]$ depending on the index space $(E,d)$ such that an $H$-fractional Brownian field indexed by $E$ exists if and only if $2H \leq \beta_E$. It is known that $\beta_{\mathcal{H}}=2$ for any Hilbert space $\mathcal{H}$ and the value of $\beta_M$ has been found for symmetric spaces such as the spheres, and the hyperbolic spaces ($\beta_{\mathbb{S}^d}=\beta_{\mathbb{H}^d}=1$) (see \cite{istas2011manifold}), through harmonic analysis techniques (see \cite{gangolli67}, \cite{molchan67}, \cite{faraut74}) or direct geometric constructions (see \cite{levy1959sphere}, \cite{chentsov1957}, \cite{takenaka1981brownian}, \cite{takenaka87}, and \cite{lifshits1980representation}). Moreover the fractional index of the cylinder $\mathbb{S}^1\times \mathbb{R}$ is zero (see \cite{Venet_cylinder}). These questions have been of interest since Lévy who prove the existence of the Brownian motion indexed by the Euclidean spaces in \cite{levyprocessus1948} and by the sphere in \cite{levy1959sphere}. However in the general case of a manifold $M$ endowed with a Riemannian metric there are few results: Istas \cite{istas2011manifold} showed that $\beta_M \leq 2$ in this case, while in the more general setting of a length-space, Feragen \textit{et al}. \cite{feragen} have proven that $\beta_M=2$ if and only if $M$ is flat. Since the fractional exponent of the circle $\beta_{\mathbb{S}^1}$ is equal to $1$, the existence of minimal closed geodesics in $M$ yields $\beta_M \leq 1$. Furthermore Morozova and Chentsov \cite{morozova68} gave a necessary condition to have $\beta_M=1$ in this case.\\

We give in this paper a geometric necessary condition for the existence of the $H$-fractional Brownian field $X^H$ indexed by a Riemannian manifold $(M,d_M)$. It appears when there exist some points $P_1\cdots,P_n$ and coefficients $c_1,\cdots,c_n$ with $\sum_{i=1}^n c_i=0$ such that $\sum_{i=1}^n c_i X^H_{P_i}=0 ~ a.s.$ (which is equivalent to $\sum_{i,j=1}^n c_i c_j d^{2H}(P_i,P_j)=0$). For the $H$-fractional Brownian field to exist in such a configuration the minimal geodesics connecting the $\left(P_i\right)_{i=1}^n$ together must span the whole tangent space with their speeds at any $P_i$ (see Theorem 1 for a rigorous statement).

On the circle $\mathbb{S}^1$ we have infinitely many configurations of four points verifying those conditions for $H=1/2$, given by any couple of pairs of antipodal points. This allows us to give a stronger necessary condition for the existence of Lévy Brownian field indexed by a manifold with at least a minimal closed geodesic (Theorem \ref{thmcercle}). This condition turns out to be very strong: in particular those manifolds cannot possess a loop of minimal length amongst all loops with nontrivial free homotopy (Theorem \ref{thmlacetminimal}). From this we conclude that any compact manifold admitting a Lévy Brownian field is simply connected (Theorem \ref{thmcompact}). Let us highlight that this result relies on the topology of the manifold, independently of the chosen Riemannian metric. We also give several examples of general situations where our result show the nonexistence of Lévy Brownain field in Section \ref{sec:ex}. 

Contrapositive of Theorem 1 gives a geometric proof of the nondegenerescence of the Lévy Brownian field on the real hyperbolic spaces (Theorem \ref{thmnondegen}).

\paragraph{Structure of the article} We start by some generalities and detail our motivations in Section \ref{sec:generalities}, give the main result in Section \ref{sec:main} while Section \ref{sec:closedminimalgeodesic} covers the case of manifolds with minimal closed geodesics and Section \ref{sec:nondegen} the nondegenerescence of Lévy Brownian field indexed by hyperbolic spaces.

\section{Generalities} \label{sec:generalities}

In this article we consider metric spaces $(E,d)$ and study the negative definiteness property for the functions $d^{2H}(x,y)$, where $H$ is a positive parameter.

The metric spaces we consider are Riemannian manifolds endowed with their geodesic distance, and in particular Riemannian manifold with at least a minimal closed geodesic.

In practice, we are looking for the \emph{fractional index} $\beta_E$ of the metric space, which is defined as the supremum of the parameters $H$ such that $d^{2H}$ is negative definite. This index is of particular interest because the function $d^{2H}$ is negative definite if and only if
\begin{equation} 2H \leq \beta_E. \end{equation}
This problematic is motivated by existence problems for fractional Brownian fields and stationary random fields indexed by $(E,d)$, which depend on the negative definiteness of $d^{2H}$. This property also gives the positive definiteness of kernels that are crucial for machine learning of nonlinear data.

In this section we recall some generalities and give details about these motivations.

\paragraph{Positive and negative definite kernels} Given a set $S$, we say that a symmetric function $f:S\times S \rightarrow \mathbb{R}$ is a \emph{positive definite kernel} if for every $x_1,\cdots,x_n \in S$ and every $\lambda_1,\cdots,\lambda_n \in \mathbb{R}$,
\begin{equation} \label{eq:positive_definite_kernel} \sum_{i,j=1}^n \lambda_i \lambda_j f(x_i,x_j) \geq 0. \end{equation}

Positive definite kernels are the covariances of random fields indexed by $S$. In particular, there exists a centred Gaussian random field indexed by $S$ with covariance $f$ if and only if $f$ is a positive definite kernel (see for instance \cite{Lifshitsbook}). Furthermore they are a key ingredient to machine learning of nonlinear data, as the positive definiteness of $f$ is equivalent to the existence of an Hilbert space $\mathcal{H}$ and a map $\Phi : S \rightarrow \mathcal{H}$ (the ``feature map") such that
\begin{equation} \label{eq:feature_map} f(x,y)=\langle \Phi(x),\Phi(y) \rangle_{\mathcal{H}}, \end{equation}
which guaranties that $f$ can play the role of a scalar product to allow for every linear machine learning method (see \cite{scholkopf2002learning}).

Positive definite kernels are closely related to negative definite kernels (see for example \cite{Berg_al}): a symmetric function $f$ is said to be a \emph{negative definite kernel} if for every $x_1,\cdots,x_n \in S$ and every $c_1,\cdots,c_n \in \mathbb{R}$ such that $\sum_{i=1}^n c_i = 0$,
\begin{equation} \label{eq:negative_definite_kernel} \sum_{i,j=1}^n c_i c_j f(x_i,x_j) \leq 0. \end{equation}

\paragraph{Fractional Brownian fields}
Given a metric space $(E,d)$ and $H>0$, we recall that an \emph{$H$-fractional Brownian field} indexed by $E$ is a centred, real-valued, Gaussian random field $(X_x)_{x\in E}$ such that
\begin{equation} \label{eq:defgen} \forall x,y \in E, ~ \E \left( X_x-X_y\right)^2=\left[d(x,y)\right]^{2H}. \end{equation}

This definition does not yield uniqueness (in law) of the field. Indeed for $N$ a centred Gaussian random variable, if $(X_t)$ is an $H$-fractional Brownian field indexed by $E$ then so is $(N+X_t)$. It is classical to define for any point $O \in E$ the \emph{$H$-fractional Brownian field with origin in $O$} by requiring also that $X_O$ be equal to $0$ almost surely. If it exists one can check that the covariance is then
\begin{equation} \label{cov} \E(X_xX_y)=\frac{1}{2}\left( d^{2H}(O,x)+d^{2H}(O,y)-d^{2H}(x,y)\right),\end{equation}
hence the uniqueness of the law of the field. Moreover the existence of the fractional Brownian field with origin in $O$ is equivalent to the positive definiteness of \eqref{cov}. A theorem of Schoenberg (see for example \cite{istas2011manifold}) proves that it is the case if and only if $d^{2H}$ is a negative definite kernel. Notice that this property does not depend on the origin $O$, and that any Gaussian field verifying \eqref{eq:defgen} is obtained by addition of a normal random variable to  an $H$-fractional with origin in an arbitrary $O \in E$: the negative definiteness of $d^{2H}$ is equivalent to the existence of every $H$-fractional Brownian field indexed by $(E,d)$.

\begin{Rem} \label{Rem:stable} In \cite{istas2011manifold} Istas define an $\alpha$-stable $H$-fractional field indexed by a metric space. Unlike in the Gaussian case, positive definiteness of the covariance is not sufficient to guaranty the existence of this field, but it is still necessary that $d^{2H\alpha}$ be negative definite: studying the negative definiteness of the powers of $d$ is also a first step for fractional non Gaussian modelling.
\end{Rem}

\paragraph{Stationary kernels} Furthermore when $d^{2H}$ is negative definite, for every \emph{completely monotone function} $F: \mathbb{R}^+\rightarrow \mathbb{R}^+$,
\begin{equation} \label{eq:statio_kernels} (x,y) \mapsto F\left(d^{2H}(x,y)\right)\end{equation}
is a positive definite kernel (see for instance \cite{Berg_al}). Let us recall that a function $F$ is completely monotone if and only $(-1)^n F^{(n)}(t) \geq 0$ for every $t\in \mathbb{R}^+$ and $n\in \mathbb{N}$.
Since the kernels in \eqref{eq:statio_kernels} depend only on the distance, they are the covariances of \emph{stationary} Gaussian random fields. These are first-choice random models for functions over $E$, whose random behaviour is homogeneous with respect to the geometry of $(E,d)$.

Positive definite kernels that are functions of a distance are also of crucial importance in kernel machine learning, since by replacing a scalar product in learning methods a kernel plays the role of a ``proximity measure". Examples of completely monotone functions include $t\mapsto e^{-\lambda t}$ for every positive $\lambda$. In particular, when they exist $e^{-\lambda d(x,y)}$ and $e^{-\lambda d^2(x,y)}$ generalise the exponential and the Gaussian kernel families.

\paragraph{Fractional index} It is a striking fact that for every metric space $(E,d)$ there exists $\beta_E$ in $[0,+\infty]
$ such that for every positive $H$,  $d^{2H}$ is negative definite if and only if (see Istas \cite{istas2011manifold})

\begin{equation}\label{eq:fractionalindex} 2H\leq \beta_E.\end{equation}

The number $\beta_E$ is called the fractional index of $(E,d)$ and is in general not easy to compute. Let us stress out some general facts which follow directly from the definition of $\beta_E$, and that we will use later.

\begin{Rem} \label{Rem:subspace} Given a metric space $(E,d)$ and $F \subset E$, if we consider $F$ as a metric space endowed with the restriction $d_{|F}$ of the distance $d$ to $F$, we have $\beta_F \geq \beta_E$.
\end{Rem}
\begin{Rem} \label{Rem:homo} For a positive $\lambda$, multiplying the distance on $E$ by $\lambda$ does not change the fractional index $\beta_E$.
\end{Rem}

\paragraph{Riemannian manifolds}
The metric spaces we consider in this article are Riemannian manifolds endowed with their geodesic distance.  Following \cite{gallot} we implicitly assume that the Riemannian manifolds we consider are $C^\infty$, connected, and countable at infinity manifolds in this whole document. Furthermore we assume the manifolds to be connected, without boundary and complete.

We will denote by $\langle ~,~\rangle$ the Riemannian product on a Riemannian manifold $M$,  $L(c)$ the length of any path with values in $M$, or any curve of $M$. $D$ is the covariant derivative associated to the Riemannian product. Finally we denote by $d_M$ the geodesic distance on $M$.

Following \cite{gallot} we recall some definitions and basic facts about geodesics.
\begin{itemize}
	\item We call a \emph{geodesic} any differentiable path with values in $M$ verifying the geodesic equation $D_{g'(t)}g'(t)=0$. We recall that such a curve is locally length minimising between its points (the reciprocal is true).  A geodesic is necessarily $C^\infty$ and parametrised with constant speed. When unspecified we implicitly consider the parametrisation by arc-length but we will use other when convenient.
	\item We call a \emph{minimal geodesic} any piecewise continuously differentiable path that is of minimal length (amongst every continuously differentiable paths)  between any pair of its points. Such a path is automatically a geodesic.
	\item Finally we call a \emph{loop} any continuous path $\gamma : [0,T] \rightarrow M$ such that $\gamma(0)=\gamma(T),$ and a \emph{minimal closed geodesic} any loop $\gamma$ such that for every points $P,Q$ on $\gamma$ there exists a minimal geodesic joining $P$ to $Q$ that is included in $\gamma$.
\end{itemize}

\begin{Rem} \label{Rem:submanifold} Given a Riemannian manifold $M$ and a submanifold $N$ of $M$, it is possible to consider the restriction $d_{M|N}$ of the geodesic distance $d_M$ to $N$. On the other hand, one can consider the Riemannian manifold $N$ endowed with the restriction $\langle ~,~ \rangle_{M|N}$ of the inner product of $M$ to $N$, which gives a geodesic distance $d_N$. In general those two distances are not equal, because the minimal geodesics in $M$ from points of $N$ take values in the whole of $M$. In particular it is not possible to deduce the value of the fractional index $\beta_M$ from local aspects of $M$ only, in spite of Remark \ref{Rem:subspace}.\end{Rem}
\begin{Rem} If $\gamma$ is a minimal closed geodesic in $M$, the two distances $d_{M|\gamma}$ and $d_{\gamma}$ are equal, and $\gamma$ is isometric to a circle of length $L(\gamma)$. In particular for a Riemannian manifold with a minimal closed geodesic, $\beta_M \leq \beta_{\mathbb{S}^1}=1$ (see Remarks \ref{Rem:subspace} and \ref{Rem:homo}) \end{Rem}

\paragraph{Free homotopy of loops} Theorem \ref{thmlacetminimal} and Theorem \ref{thmcompact} involve free homotopy of loops.
\begin{Def}[Free homotopy of loops] 
	We say that two loops $\gamma_1$, $\gamma_2$ are \emph{freely homotopic} if there exists reparametrisations $\tilde{\gamma}_1,\tilde{\gamma}_2:[0,1]\rightarrow M$ and a homotopy of loops from $\tilde{\gamma}_1$ to $\tilde{\gamma}_2$, that is to say  a continuous map
	\begin{alignat*}{1} f : [0,1] & \times [0,1] \rightarrow M \\
	( & s, t) \longmapsto f_t(s) \end{alignat*}
	
	such that $f_0=\tilde{\gamma}_1$ and $f_1=\tilde{\gamma}_2$ and for every $t\in[0,1], ~ f_t : [0,1] \rightarrow M$ is a loop.
	In this case we write $ \gamma_1 \sim \gamma_2.$
\end{Def}

We refer to \cite{Hatcher} for generalities about free homotopy of loops. We will use the fact that $\sim$ is an equivalence relation.

\section{Main result} \label{sec:main}
We give in this section the main result of the article and its proof. All our other results follow from Theorem \ref{thmmainresult}.
\subsection{Statement and proof}
We start by giving some definitions. The first two are directly connected to the property of negative definiteness for the power $d^{2H}$ of a distance.

\begin{Def}[Configurations] \label{Def:conf} Given a metric space $(E,d)$, we call a \emph{configuration} $((P_1,\cdots,P_n),(c_1,\cdots,P_n))$ any finite collection of distinct points $(P_1,\cdots,P_n)\in E^n$ with $(c_1,\cdots,c_n) \in \left(\mathbb{R}^*\right)^n$ such that $$\sum_{i=1}^n c_i = 0.$$
	\end{Def}
	\begin{Def}[Critical configurations]
	Given $H>0$, we say that a configuration is \emph{$H$-critical} if $$\sum_{i,j=1}^n c_i c_j d^{2H}(P_i,P_j)=0.$$
	\end{Def}
\begin{Rem} \label{remcritical} Let us observe that if there exists an $H$-fractional Brownian field $X^H$ indexed by $(E,d)$ we have
	
	$$\sum_{i,j=1}^n c_i c_j d^{2H}(P_i,P_j)=\Var \left( \sum_{i=1}^n c_i  X^H_{P_i} \right).$$
	
	 In this case the configuration $((P_1,\cdots,P_n),(c_1,\cdots,c_n ))$ is $H$-critical if and only if $$\sum_{i=1}^n c_i  X^H_{P_i}=0 \text{ \emph{almost surely.}}$$\end{Rem}
	
	\begin{Def}[Space of shortest directions] Given a Riemannian manifold $M$, $P\in M$ and $S \subset M$, we define the \emph{space of shortest directions from $P$ to $S$}
		
		$$ T_{P\rightarrow S}=\vectspan \left\{\begin{array}{cc} g'(0) ~|~ \exists Q \in S, ~ g:[0,1]\rightarrow M \\ \text{ minimal geodesic from  $P$ to $Q$}\end{array} \right\},$$
		where $\vectspan(V)$ denotes the linear span of a set of vectors $V$: the space of shortest directions $T_{P\rightarrow S}$ is a vector subspace of the tangent space $T_P(M)$.
	\end{Def}
Let us state the main result of the chapter:
\begin{Thm} \label{thmmainresult} Let $(M,d_M)$ be a complete  Riemannian manifold and $H$ in $]0,1[$. If there exists an $H$-fractional Brownian field indexed by $M$, then for every $H$-critical configuration $((P_1,\cdots,P_n),(c_1,\cdots,c_n))$,

	\begin{equation}  \label{G} \tag{G} \forall ~ i \in \{1,\cdots,n\}, ~ \dim T_{P_i\rightarrow \{P_j, j \neq i\} } = \dim M. \end{equation}
	\end{Thm}
		 To prove Theorem \ref{thmmainresult} we will assume the existence of an $H$-critical configuration such that $\eqref{G}$ does not hold and show that there exists no $H$-fractional Brownian field indexed by $M$. 		

		  \begin{proof}Let us consider $H$ in $]0,1[$ and assume the existence of an $H$-critical configuration, that is to say distinct points $P_1,\cdots,P_n \in M$ and $c_1,\cdots,c_n \in \mathbb{R}^*$ such that $$\sum_{i=1}^n c_i=0$$ and $$\sum_{i,j=1}^n c_i c_j d_M^{2H}(P_i,P_j)=0.$$ Furthermore we suppose that the points $P_1,\cdots, P_n$ do not verify the geometrical condition \eqref{G}, therefore
		  $$ \exists i \in \{1, \cdots, n\},  ~ \dim T_{P_i\rightarrow \{P_j, j \neq i\} } < \dim M.$$
		 Without loss of generality we assume $i=n$ and consider a geodesic $g_\perp$ parametrised by arc-length with $g_\perp(0)=P_n$ and $g_\perp'(0) \in \left(T_{P_n\rightarrow \{P_j, j \neq n\} }\right)^\perp$.
Using Lemma \ref{lemtek2} we obtain a sequence of positive $\varepsilon_m$ converging towards zero and geodesics $(g_i)_{1\leq i \leq n-1}$ such that for every $i$ in $\{1,\cdots,n-1\},$
			$$ d_M(P_i, g_\perp(\varepsilon_m))= d_M (P_i, P_n) + \left\langle g_\perp'(0), g_i'(d_M(P_i,P_n))\right\rangle_M \varepsilon_m + O \left( \varepsilon_m^2\right).$$
			
			Since $g_\perp'(0) \in \left(T_{P_n\rightarrow \{P_j, j \neq n\} }\right)^\perp$  we get for every $1\leq i \leq n -1$
			 $$\left\langle g_\perp'(0), g_i'(d_M(P_i,P_n))\right\rangle_M=0,$$
			Hence \begin{equation} \label{eq:DL} d_M(P_i, g_\perp(\varepsilon_m))= d_M (P_i, P_n) + O \left( \varepsilon_m^2\right), \end{equation}
from which
\begin{alignat*}{1}
d_M^{2H}(P_i,g_\perp(\varepsilon_m)) &= \left(d_M(P_i,P_n)+ O\left( \varepsilon_m^2\right)\right)^{2H} \\
&= d_M^{2H}(P_i,P_n)\left( 1+\frac{O\left( \varepsilon_m^2\right)}{d_M(P_i,P_n)} \right)^{2H} \\
& =  d_M^{2H}(P_i,P_n)\left( 1+2H\frac{O\left( \varepsilon_m^2\right)}{d_M(P_i,P_n)} + O \left( \varepsilon_m^4\right)\right).
\end{alignat*}
Using $H<1$ we obtain
\begin{equation} \label{DL1} d_M^{2H}(P_i,g_\perp(\varepsilon_m)) = d_M^{2H}(P_i,P_n) + o\left(\varepsilon_m^{2H}\right). \end{equation}

Let us also notice that for $m$ large enough, $\varepsilon_m$ is small enough so that the geodesic $g_\perp$ is minimal between $g_\perp(\varepsilon_m)$ and $P_n$ (see \cite{gallot}), hence
\begin{equation} \label{DL2} d_M^{2H}(P_n,g_\perp(\varepsilon_m))=\varepsilon_m^{2H}.\end{equation}
 
 We now set \begin{equation} \label{eq:Pn+1}P_{n+1}:=g_\perp(\varepsilon_m)\end{equation} and consider the configuration $((P_1,\cdots,P_n,P_{n+1}),(c'_1,\cdots, c'_{n+1}))$, with \\ $(c'_1,\cdots,c'_{n-1},c'_n,c'_{n+1})=(c_1,\cdots,c_{n-1},c_n/2,c_n/2)$ so that in particular $$ \sum_{i=1}^{n+1} c'_i = \sum_{i=1}^{n} c_i= 0.$$  Let us compute

 \begin{multline*} \shoveleft{\sum_{i,j=1}^{n+1}c'_i c'_j d_M^{2H}(P_i,P_j)} \\
 \shoveleft{=  \sum_{i,j=1}^{n-1}c_i c_j d_M^{2H}(P_i,P_j) + 2 \sum_{i=1}^{n-1} c_i \frac{c_n}{2} d_M^{2H}(P_i,P_n)+2 \sum_{i=1}^{n-1} c_i \frac{c_n}{2} d_M^{2H}(P_i,P_{n+1})}\\
 \shoveright{+2\left(\frac{c_n}{2}\right)^2 d_M^{2H}(P_n,P_{n+1}).} \\
~
 \end{multline*}
 Recalling \eqref{eq:Pn+1} we use \eqref{DL1} and \eqref{DL2} to obtain
 \begin{alignat*}{1}
 & \sum_{i,j=1}^{n-1}c_i c_j d_M^{2H}(P_i,P_j) + 2 \sum_{i=1}^{n-1} c_i \frac{c_n}{2} d_M^{2H}(P_i,P_n) +2 \sum_{i=1}^{n-1} c_i \frac{c_n}{2} \left( d_M^{2H}(P_i,P_n) + o\left(\varepsilon_m^{2H}\right) \right)\\ & +2 \left(\frac{c_n}{2}\right)^2 \varepsilon_m^{2H} \\
 &= \sum_{i,j=1}^{n}c_i c_j d_M^{2H}(P_i,P_j) + c_n \sum_{i=1}^{n-1} c_i ~ o\left(\varepsilon_m^{2H}\right) +\frac{c_n^2}{2} \varepsilon_m^{2H}.\\
 \end{alignat*}

 By hypothesis $$\sum_{i,j=1}^{n}c_i c_j d_M^{2H}(P_i,P_j)=0,$$
 hence it is clear that
 $$\sum_{i,j=1}^{n+1}c'_i c'_j d_M^{2H}(P_i,P_j) = \frac{c_n^2}{2}  ~\varepsilon_m^{2H} + o\left(\varepsilon_m^{2H}\right)$$
 is positive for $m$ large enough. We conclude that $d^{2H}$ is not of negative type and therefore there exists no $H$-fractional Brownian field indexed by $(M,d_M)$. \end{proof}
 			\begin{Rem} To prove Theorem \ref{thmmainresult} we have exhibited a configuration \\ $((P_1,\cdots,P_{n+1}),(c'_1,\cdots,c'_{n+1}))$ such that $$\sum_{i=1}^{n+1} c'_i c'_j d^{2H}(P_i,P_j) > 0. $$ To obtain it from the critical configuration $((P_1,\cdots,P_{n}),(c_1,\cdots,c_n))$ we only added $P_{n+1}$ as close as wanted to $P_n$ and the coefficients $c_i$ remained the same, except for $P_n$ which ``loses half of its coefficient to $P_{n+1}$", that is to say $$c'_n=c'_{n+1}=c_n/2.$$ We can look at this new configuration as an infinitesimal perturbation of the critical configuration. Condition $\eqref{G}$ from Theorem \ref{thmmainresult} is necessary to avoid that those perturbations of a critical configuration prevent the fractional Brownian motion to exist. Let us observe that while the perturbation happens in a neighbourhood of $(P_1,\cdots,P_n)$ it is impossible to decide whether $P_1,\cdots,P_n \in M$ verify $\eqref{G}$ without considering the whole manifold $(M,d)$, because condition $\eqref{G}$ deals with minimal geodesics (see also Remark \ref{Rem:submanifold}).
 			\end{Rem}

\section{Manifolds with a minimal closed geodesic} \label{sec:closedminimalgeodesic}
In this section we apply our main result to a Riemannian manifold with at least a minimal closed geodesic. In this context we obtain Theorem \ref{thmcercle}, which is a stronger version of Theorem \ref{thmmainresult}, that we apply to give several examples and corollaries.

In particular Theorem \ref{thmcompact} states that a compact manifold admitting a Lévy Brownian field is necessary simply connected.
\subsection{A stronger version of Theorem \ref{thmmainresult}}
Let $(M,d_M)$ be a Riemannian manifold with a minimal closed geodesic $\gamma$.

\begin{Def}[Antipodal point]
For any $P$ on $\gamma$, we call the antipodal point of $P$ on $\gamma$, denoted by $P^*$, the unique point of $\gamma$ such that $$d_M(P,P^*)=L(\gamma)/2.$$
\end{Def}
\paragraph{Many $1/2$-critical configurations on $\gamma$}
Consider now distinct points $P_1,\cdots,P_4\in \gamma$ such that 
\begin{equation} \label{antipod1} P_3=P_1^*\end{equation} and \begin{equation} \label{antipod2} P_4=P_2^*. \end{equation}

\begin{figure}[h!]
	\centering \def\svgwidth{100pt}
	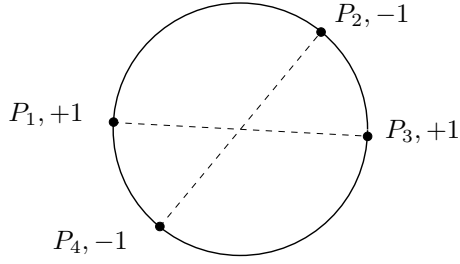 \caption{An $1/2$-critical configuration on the circle}
	\label{fig:conf}
\end{figure}

Because $\gamma$ is a minimal closed geodesic we know that $d_M$ restricted to $\gamma$ is the distance on $\mathbb{S}^1$ up to a multiplication by $L(\gamma)/2\pi$. Setting \begin{equation}\label{eq:c1...c4}c_1=c_3=1, \text{ and } c_2=c_4=-1,\end{equation} it is easy to check that \begin{equation} \label{eq:som4zero} \sum_{i,j=1}^4 c_i c_j d_M(P_i,P_j)=0.\end{equation}
Observe that $c_1+c_2+c_3+c_4=0$, so that $((P_1,P_2,P_3,P_4),(c_1,c_2,c_3,c_3))$ is a $1/2$-critical configuration.

We already know that if there exists a Lévy Brownian field (\textit{i.e.} a $1/2$-fractional Brownian field) indexed by $M$, every distinct points $P_1,\cdots,P_4\in \gamma$ verifying \eqref{antipod1} and \eqref{antipod2} must verify condition $\eqref{G}$ from Theorem \ref{thmmainresult}.

Furthermore it is possible to consider $P_1,P_2,P_3,P_4$ as required with \linebreak $d_M(P_1,P_2)$ as small as wanted. Because $P_1,P_3$ and $P_2,P_4$ are respectively antipodal, $d_M(P_2,P_4)=d_M(P_1,P_3)$ is also as small as wanted. This allows us to give the following result.
\begin{Thm}\label{thmcercle}
	Let $(M,d_M)$ be a complete Riemannian manifold such that there exists a Lévy Brownian field indexed by $(M,d_M)$. Then for every minimal closed geodesic $\gamma$ and every $P\in \gamma$, 
	$$\dim T_{P\rightarrow \{P^*\} } = \dim M.$$

\end{Thm}

\begin{proof}
	Let us assume there exists $P\in \gamma$ such that $\dim T_{P\rightarrow \{P^*\} } < \dim M$. We take $P_4=P$, $P_2=P_4^*$, and choose for every $\eta \in ]0,\pi[$ a point $P_3(\eta) \in \gamma$ such that $d_M(P_4,P_3(\eta))=\eta >0$. Finally we define $P_1(\eta)=(P_3(\eta))^*$ (see Figure \ref{figcercleeta} above).  Please notice that we will sometimes write $P_i$ with $i$ taking values in $\{1,2,3,4 \}$: what we mean is $P_i(\eta)$ if $i\in\{ 1,3 \}$, and $P_i$ if $i \in \{2,4\}$.
	
	\begin{figure}[h!]
		\centering \def\svgwidth{134pt}
		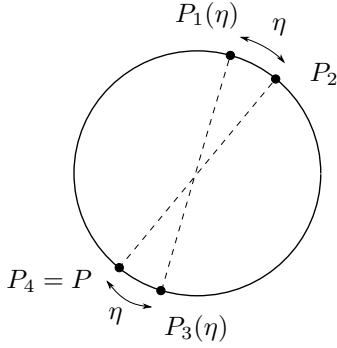 \caption{Disposition of the four points on $\gamma$}
		\label{figcercleeta}
	\end{figure}
	
	Let $g_\perp$ be a geodesic parametrised by arc-length with $g_\perp(0)=P_4$ and $g_\perp ' (0)\in \left( T_{P\rightarrow \{P^*\}} \right)^\perp$ .
	
	Using Lemma \ref{lemtek2} we get a sequence of positive $\varepsilon_m$ with $\lim \varepsilon_m = 0$ such that for each $i\in \{1,2,3\}$ there exists
	$$g_i^{\eta}:[0,d_M(P_i,P_4)] \rightarrow M$$
	 a minimal geodesic with $g_i^{\eta}(0)=P_i$, $g_i^{\eta}(d_M(P_i,P_4))=g_\perp(0)=P_4$ and such that
	\begin{equation} d_M(P_i, g_\perp(\varepsilon_m))=d_M(P_i,P_4)+ \left\langle g_\perp ' (0), \left( g_i^\eta \right)'(d_M(P_i,P_4)) \right\rangle_M \varepsilon_m+O_\eta\left( 
	\varepsilon_m ^2\right). \end{equation}
	
 Let us remark that the above expression is not uniform in $m$. Indeed the sequence $(\varepsilon_m)$ depends on $\eta$, as much as the sequence $m \mapsto O_\eta \left(\varepsilon_m \right)$. This is not a problem as we fix $\eta$ before we pass to the limit in $m$ at the end of the proof.
	
We now distinguish three cases:
	
	\begin{itemize}
		\item i=1: Let us consider the following reparametrisations of the $g^\eta_1$:
			\begin{alignat*}{1}
				 \tilde{g}^\eta_1  & : [0,1] \rightarrow M \\
				& ~~~~~ t \mapsto g_1^\eta \big(t d_M\big(P_1(\eta),P_4\big)\big).
			\end{alignat*}
			Applying Lemma \ref{lemtek1} we get a sequence of $\eta_n>0$ converging towards zero such that $\tilde{g}_1^{\eta_n}$ converges to $\tilde{g}_1$ minimal geodesic between $\lim\limits_{n\rightarrow +\infty} P_1(\eta_n)= P_2$ and $P_4$ when $n$ goes to infinity. We also have \linebreak $\lim\limits_{n\rightarrow \infty}\left(\tilde{g}_1^{\eta_n}\right)'(1)=\tilde{g}_1'(1)$ hence
			\begin{alignat*}{1}
			\left(g_1^{\eta_n}\right)'\big(d_M\big(P_1(\eta_n),P_4\big)\big)=\frac{\left(\tilde{g}_1^{\eta_n}\right)'(1)}{d_M\big(P_1(\eta_n),P_4\big)} \underset{n\rightarrow +\infty}{\rightarrow} \frac{\tilde{g}_1'(1)}{d_M(P_2,P_4)}.
			\end{alignat*}
			We obtain $$\lim\limits_{n\rightarrow \infty}\Big\langle g_\perp'(0), \left(g^{\eta_n}_1\right)'  \left( d_M\big(P_1(\eta_n),P_4\big) \right) \Big\rangle_M=\frac{\left\langle g_\perp'(0), \tilde{g}_1'  (1) \right\rangle_M}{d_M(P_2,P_4)}=0,$$ 
			because $\tilde{g}_1$ is a minimal geodesic from $P_2$ to $P_4$ hence $\tilde{g}_1'  (1)\in T_{P\rightarrow \{P^*\}} $.
		
		\item i=2: $\left\langle g_\perp'(0), \left(g^{\eta}_2\right)'  (d_M(P_2,P_4)) \right\rangle_M=0$ because $$\left( g^{\eta}_2\right)'  (d_M(P_2,P_4))  \in T_{P\rightarrow \{P^*\}}.$$
		\item i=3: We again consider reparametrisations $\tilde{g}_3^{\eta_n} : [0,1] \rightarrow M$ of the $g_3^{\eta_n}$ and apply Lemma \ref{lemtek1} to obtain the convergence of $\tilde{g}_3^{\eta_n}$ towards a minimal geodesic $\tilde{g}_3$ between $P_3(\eta_n)$ and $P_4$. For $n$ large enough, $d_M(P_3(\eta_n),P_4)=\eta_n$ is small enough so that there exists a unique minimal geodesic between $P_3(\eta_n)$ and $P_4$ (up to reparametrisations). This proves that $\tilde{g}_3$ is included in $\gamma$ when $n$ is large enough hence $\tilde{g}_3'(1) \in T_{P\rightarrow \{P^*\}}$. Proceeding with the same computations we did for $i=1$ we obtain
		$$ \lim\limits_{n\rightarrow \infty}\left\langle g_\perp'(0), \left(g^{\eta_n}_3\right)' \left(d_M\big(P_3(\eta_n),P_4\big)\right) \right\rangle_M=0. $$
	\end{itemize}
	In the end for every $i$ in $\{1,2,3 \}$ we have \begin{equation} \label{limn} \displaystyle \lim\limits_{n\rightarrow \infty} \left\langle g_\perp ' (0), \left( g_i^{\eta_n} \right)'(d_M(P_i,P_4)) \right\rangle_M  =0. \end{equation}
	
	We now follow exactly the proof of Theorem \ref{thmmainresult}. Setting $P_5=g_\perp(\varepsilon_m)$ and $(c'_1,\cdots,c'_5)=(c_1,c_2,c_3,c_4/2,c_4/2)$. Recall from \eqref{eq:c1...c4} and \eqref{eq:som4zero} that $(c_1,\cdots,c_4)=(1,-1,1,-1)$ so that $\displaystyle \sum_{i,j=1}^4 c_ic_j d(P_i,P_j)=0$.  We obtain
	\begin{alignat*}{1}
		& \sum_{i,j=1}^5 c'_i c'_j d_M(P_i,P_j) \\
		= & \sum_{i,j=1}^{3}c_i c_j d_M(P_i,P_j) + 2 \sum_{i=1}^{3} c_i \frac{c_4}{2} d_M(P_i,P_4)  \\
		& +2 \sum_{i=1}^{3} c_i \frac{c_4}{2} \Big{[}d_M(P_i,P_4)+\left\langle g_\perp ' (0), \left( g_i^{\eta_n} \right)'(d_M(P_i,P_4)) \right\rangle_M  \varepsilon_m+O_{\eta_n}\left( \varepsilon_m ^2\right) \Big{]}\\
		&+2 \left(\frac{c_4}{2}\right)^2 \varepsilon_m \\
		= & \underbrace{\sum_{i,j=1}^{4}c_i c_j d_M(P_i,P_j) }_{=0}+ c_4\sum_{i=1}^3 c_i \left( \left\langle g_\perp ' (0), \left( g_i^{\eta_n} \right)'(d_M(P_i,P_4)) \right\rangle_M \varepsilon_m +O_{\eta_n}\left( \varepsilon_m ^2\right) \right)\\
		&+ 2 \left(\frac{c_4}{2}\right)^2 \varepsilon_m \\
		= & \varepsilon_m \left( \frac{1}{2} - \sum_{i=1}^3 c_i \left\langle g_\perp ' (0), \left( g_i^{\eta_n} \right)'(d_M(P_i,P_4)) \right\rangle_M \right) + O_{\eta_n}\left( \varepsilon_m ^2\right).
	\end{alignat*}
	Using \eqref{limn} if we fix $n$ large enough we have $$ \frac{1}{2} - \sum_{i=1}^3 c_i \left\langle g_\perp ' (0), \left( g_i^{\eta_n} \right)'(d_M(P_i,P_4)) \right\rangle_M>0.$$ In this case for $m$ large enough $$ \sum_{i,j=1}^5 c'_i c'_j d_M(P_i,P_j) >0.$$ We conclude that $d_M$ is not of negative type and therefore there exists no Lévy Brownian field indexed by $(M,d_M)$.
\end{proof}

\begin{Rem} In \cite{morozova68} Chentsov and Morozova give a different necessary condition for the existence of a Lévy Brownian field indexed by a manifold with a minimal closed geodesic. 	The proof of their result is based on the remark that for any Lévy Brownian field $X_P$ indexed by $\mathbb{S}^1$, $$ \forall P \in \mathbb{S}^1,~  X_P+X_{P^*} = 2 \int_{\mathbb{S}^1} X_Q dQ, $$ where $dQ$ denotes the uniform measure on $\mathbb{S}^1$. In particular this random variable does not depend on $P$. This directly implies that $$ \forall P,P' \in \mathbb{S}^1, ~X_P+X_{P^*}-X_{P'}-X_{P'*}=0 ~\text{almost surely.}$$ Equivalently $((P,P^*,Q,Q^*),(1,-1,1,-1))$ is a $1/2$-critical configuration (see Remark \ref{remcritical}), which is our starting point to prove Theorem \ref{thmcercle}. However their statement seems distinct from ours. Furthermore using Theorem \ref{thmcercle} we show there exists no Lévy Brownian field in cases where it seems uneasy to check whether Morozova and Chentsov's condition is verified or not (See Examples \ref{Ex:outsidesurface} and \ref{Ex:rotation_invariant}, Theorem \ref{thmlacetminimal} and \ref{thmcompact}).
\end{Rem}

\subsection{Examples} \label{sec:ex}

\begin{Exe}[Surface outside a sphere tangent to a great circle] \label{Ex:outsidesurface} Let us denote by $B(0,1) \subset \mathbb{R}^3 $ the closed ball of center $0$ and radius $1$, and $C$ a great circle of the unit sphere $\mathbb{S}^2$. Let $S \subset \mathbb{R}^3$ be a surface such that $S \cap B(0,1)=C$. Let us show that there exists no Lévy Brownian field indexed by $S$.
	
	\begin{figure}[h!]
		\centering \def\svgwidth{150pt}
		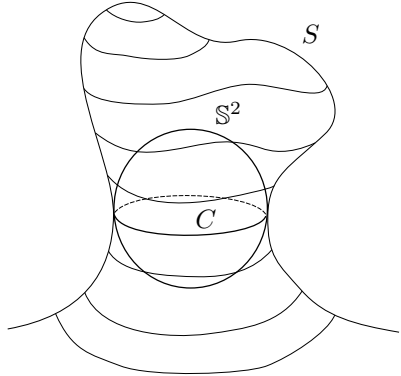 \caption{A surface $S$ verifying our assumptions }
		\label{fig:surface_ext}
	\end{figure}
	
We consider \begin{alignat*}{1}  \Pi : \mathbb{R}^3 \rightarrow \mathbb{R}^3\\
																X \mapsto \frac{X}{\| X\|}. \end{alignat*}

\begin{Lem} \label{LemPi} For every curve $c$ with values in $S\subset \mathbb{R}^3$,
	$$ L(c) \geq L\left( \Pi \circ c \right),$$
	with equality if and only if $c$ takes values in $C$.
\end{Lem}

\begin{proof} For every $X=(x,y,z) \in \mathbb{R}^3$ we have $$\Pi (X)=\frac{(x,y,z)}{(x^2+y^2+z^2)^{1/2}}.$$
	The map $\Pi$ is differentiable on $\mathbb{R}^3 \setminus \{0 \}$. For every $X \in \mathbb{R}^3 \setminus \{0 \}$ we compute the Jacobian matrix
	$$ D \Pi _X= \frac{1}{\|X\|}\left(\begin{array}{ccc}
	1-\frac{x^2}{\|X\|^2} & \frac{xy}{\|X\|^2}& \frac{xz}{\|X\|^2}\\
	\frac{yx}{\|X\|^2} & 1-\frac{y^2}{\|X\|^2} & \frac{yz}{\|X\|^2} \\
	\frac{zx}{\|X\|^2} & \frac{zy}{\|X\|^2}& 1- \frac{z^2}{\|X\|^2}
	\end{array}\right).$$
	We have $$ D \Pi _X= \frac{1}{\|X\|} \left( I_3 -n^{\mathsf{T}}n \right),$$
	with $$n=\overrightarrow{\nabla} \| X \|=\frac{1}{\|X\|}\left(\begin{array}{c}
	x\\
	y\\
	z
	\end{array}\right).$$
	Let us notice that $ I_3 -n^{\mathsf{T}}n $ is the matrix of the orthogonal projection on the plane of normal vector $n$. In particular for every $U \in \mathbb{R}^3$,  $$\|(I_3 -n^{\mathsf{T}}n) U\| \leq \| U \|,$$
	hence for every $X$ such that $\|X \|\geq 1$ and $U \in \mathbb{R}^3$,
	$$ \| D \Pi _X U \| \leq \|U\| ,$$
	with equality if and only if $\| X \|=1$ and $U$ is tangent to the unit sphere \nolinebreak $\mathbb{S}^2$.
	
	As a consequence for a curve $c:[a,b] \rightarrow S$, for every $t\in [a,b],$ $\|c(t)\| \geq 1$ and we obtain 
	
	$$ L( \Pi \circ c)=\int_a^b \| (\Pi \circ c)'(t) \| dt = \int_a^b \| D\Pi_{c(t)} c'(t) \| dt \leq \int_a^b \| c'(t) \|dt=L(c).   \qedhere$$
\end{proof}

Let us now consider $P$ and $Q$ two points on the great circle $C$. Let $c$ be a minimal geodesic from $P$ to $Q$ in $S$. From Lemma \ref{LemPi} we have
$$ L(c) \geq L(\Pi \circ c).$$
Since $\Pi \circ c$ is a curve with values in $\mathbb{S}^2$ we know that $L(\Pi \circ c) \geq d_{\mathbb{S}^2}(P,Q).$ On the other the shorter arc of the great circle $C$ joining $P$ to $P^*$ is a curve with values in $S$ and length $d_{\mathbb{S}^2}(P,Q)$, which gives $d_{\mathbb{S}^2}(P,Q) \geq L(c)$ since $c$ is a minimal geodesic. We obtain
$$ d_{\mathbb{S}^2}(P,Q) \geq L(c) \geq L(\Pi \circ c) \geq d_{\mathbb{S}^2}(P,Q), $$
which means $L(c)=L(\Pi \circ c)$. The equality case of Lemma \ref{LemPi} states that $c$ is included in the great circle $C$. We have shown that every minimal geodesic from $P$ to $Q$ is included in $C$, which shows that $C$ is a minimal closed geodesic of $S$, together with
$$ \dim T_{P \rightarrow \{P^*\}}=1<\dim S=2$$ for every $P\in C$.

From Theorem \ref{thmcercle} we know that there exists no Lévy Brownian field indexed by $S$. 
\end{Exe}

	\begin{Exe}[rotation-invariant manifolds with a shortest parallel] \label{Ex:rotation_invariant} Let us consider a complete Riemannian manifold $N$ and consider $M=\mathbb{S}^1 \times N$ endowed with the Riemannian metric $$\langle ~, ~ \rangle_M= f(z)\langle~,~\rangle_{\mathbb{S}^1}+\langle~,~\rangle_N,$$ where $f : N \rightarrow \mathbb{R}^*_+$ is a $C^\infty$ function with a global minimum at $z_0 \in N$. We take two points $P,Q$ on the parallel $\gamma=\mathbb{S}^1 \times \{ z_0\}$ and a curve $g: [0,T] \rightarrow M$ with $g(0)=P,~g(T)=Q$. We write $g(t)=(\theta(t),z(t))$ and compute the energy of~$g:$
		\begin{alignat*}{1} E_M(g) &=\frac{1}{2}\int_0^T \langle g'(t),g'(t) \rangle_M dt \\ & = \frac{1}{2}\int_0^T f(z) \left\langle \theta'(t),\theta'(t)\right\rangle_{\mathbb{S}^1}+\left\langle z'(t),z'(t)\right\rangle_N dt \\
		& \geq  \frac{1}{2} \int_0^T f(z_0) \left\langle \theta'(t),\theta'(t)\right\rangle_{\mathbb{S}^1} dt. \end{alignat*}
		Since minimal geodesics are also minimisers of the energy (see $e.g.$ \cite{gallot}) it is clear that minimal geodesics between $P$ and $Q$ are included in $\gamma$. This shows that any proportional to arc-length parametrisation of $\gamma$ is a minimal closed geodesic together with $$\dim T_{P\rightarrow \{ P*\}}=1 < \dim M$$ for every $P\in \gamma$. We apply Theorem \ref{thmcercle} to get the nonexistence of Lévy Brownian fields indexed by $M$.\end{Exe}

\subsection{Manifolds with a shortest nontrivial loop} \label{subsec:loopresults}
In the following we denote by $\mathscr{L}(M)$ the set of all piecewise continuously differentiable loops with values in M. Notice that $L(\gamma)$ is properly defined for every $\gamma \in \mathscr{L}(M)$. We denote by $\mathscr{C}_1(M)=\mathscr{L}(M)/ \!\! \sim$ the set of all free homotopy classes of piecewise continuously differentiable loops in $M$. Furthermore we denote by $C\in \mathscr{C}_1(M)$ the class of piecewise continuously differentiable loops freely homotopic to any constant loop (recall that all the manifolds we consider are connected hence all the constant loops are freely homotopic to each others).

	\begin{Thm} \label{thmlacetminimal}
		Let $M$ be a Riemannian manifold of dimension at least $2$ such that there exists $\gamma$ of minimal length in $\mathscr{L}(M)\setminus C$.
		There exists no Lévy Brownian field indexed by $M$.
	\end{Thm}
	\begin{proof} From Lemma \ref{lemlacetminimal} we know that for every $P,Q \in \gamma$ all the minimal geodesics from $P$ to $Q$ are included in $\gamma$. In particular every proportional to arc-length parametrisation of $\gamma$ is a minimal closed geodesic such that $$\forall P \in \gamma, ~ \dim T_{P\rightarrow \{P^*\} } = 1.$$

Since $\dim(M) \geq 2$, Theorem \ref{thmcercle} shows that there exists no Lévy Brownian field indexed by $M$.
\end{proof}
For a compact manifold Cartan's theorem gives the existence of a minimal closed geodesic in every free homotopy class. We adapt its proof to give Lemma \ref{lemcompact} (see Appendix), that gives the following result.
		
\begin{Thm} \label{thmcompact}	Let $M$ be a compact, nonsimply connected Riemannian manifold of dimension at least $2$. There exists no Lévy Brownian field indexed by $M$.
\end{Thm}
\begin{proof} Lemma \ref{lemcompact} gives the existence of $\gamma$ of minimal length in $\mathscr{L}(M)\setminus C$. Since $M$ is compact it is complete as a metric space, hence a complete Riemannian manifold (see e.g. \cite{gallot}), and Theorem \ref{thmlacetminimal} applies. \end{proof}
\begin{Exe}[closed surfaces] In dimension $2$ due to the classification of closed (\textit{i.e.} compact without boundary) surfaces a closed surface admitting a Lévy Brownian field is homeomorphic to the sphere $\mathbb{S}^2$.
	\begin{Rem} Let us notice that we have examples of compact and simply connected manifolds that do not admit a Lévy Brownian field (provided by the results from Example \ref{Ex:outsidesurface} or \ref{Ex:rotation_invariant}). To the knowledge of the author, the spheres $\mathbb{S}^n$ are the only compact manifolds on which we know a Lévy Brownian field exists. \end{Rem}
\end{Exe}

\section[Nondegeneracy of fractional Brownian fields indexed by hyperbolic spaces]{Nondegeneracy of fractional Brownian fields indexed by hyperbolic spaces
\sectionmark{Nondegeneracy of hyperbolic fields}}
\sectionmark{Nondegeneracy of hyperbolic fields} \label{sec:nondegen}
Let us recall some elementary facts on the hyperbolic spaces. We again refer to \cite{gallot} for proofs and details.
\subparagraph{Poincaré ball model} There are many ways to present the hyperbolic space $\mathbb{H}^d.$ We briefly introduce the Poincaré disk model.

Let us consider $B_d$ the open ball of radius $1$ in $\mathbb{R}^d$. We endow $B_d$ with the Riemannian metric
$$ \langle ~, ~\rangle_d= \frac{4 \displaystyle \sum_{i=1}^d dx_i^2}{ \displaystyle \left(1-\sum_{i=1}^d x_i^2\right)^2}.$$
One can check that we obtain a complete Riemannian manifold of curvature $-1$, which we call the hyperbolic space of dimension $d$ and denote by $\mathbb{H}^d$. It is well known that the geodesics of $\mathbb{H}^d$ are given by the arcs of the circles which intersect orthogonally the sphere $\mathbb{S}^{d-1}$ of radius $1$ (By circles and sphere we mean: the usual circles and sphere from Euclidean geometry in $\mathbb{R}^d$).

Furthermore if we consider the unique circle of $\mathbb{R}^{d+1}$ passing through $P,Q \in B_d \times \{ 0 \}$ which is orthogonal to $\mathbb{S}^{d} \times \{0\}$, we notice it is included in $\mathbb{R}^d \times \{0\}$ at all time and orthogonal to $\mathbb{S}^{d-1}$. This shows that the minimal geodesics between points of $B_d$ are the same in $\mathbb{H}^d$ and $\mathbb{H}^{d+1}$, thus the inclusion $B_d \times \{0\} \subset B_{d+1}$ extends to an isometric immersion
$$ \mathcal{I} : \mathbb{H}^d \hookrightarrow \mathbb{H}^{d+1},$$
that is to say \begin{equation} \label{immersion} \forall P,Q \in \mathbb{H}^d, ~ d_{\mathbb{H}^d}(P,Q)=d_{\mathbb{H}^{d+1}}(\mathcal{I}(P),\mathcal{I}(Q)). \end{equation}

Finally let us recall that for all $d\geq 1$  there exists an $H$-fractional Brownian field indexed by $\mathbb{H}^d$ if and only if $0 < H \leq 1/2$ (see \cite{istas2005spherical}).

The following result is not surprising but should be quite tedious to prove with computations, if possible. Here we give a geometric proof of this fact.
\begin{Thm} \label{thmnondegen}
	\begin{enumerate}
			\item For every $0< H \leq 1/2$ there are no $H$-critical configurations in $\mathbb{H}^d$.
\item Let $0< H \leq 1/2$ and $X^H$ be an $H$-fractional Brownian field indexed by the $d$-dimensional hyperbolic space $\mathbb{H}^d$, such that there exists $O \in \mathbb{H}^d$ and $X^H_O=0$ $a.s.$.
For all distinct $P_1,\cdots,P_n \in\mathbb{H}^d$ the Gaussian vector $\left(P_1,\cdots,P_n\right)$ is nondegenerate.\end{enumerate}
\end{Thm}
\begin{proof} \begin{enumerate} \item Let us assume there exists an $H$-critical configuration \linebreak $((P_1,\cdots,P_n),(c_1,\cdots,c_n))$ of $\mathbb{H}^d$.
Using \eqref{immersion} it is clear that
\begin{alignat*}{1} &\sum_{i,j=1}^n c_i c_j d_{\mathbb{H}^d}^{2H} (P_i,P_j)=0 \\ \Rightarrow & \sum_{i,j=1}^n c_i c_j \left[d_{\mathbb{H}^{d+1}}(\mathcal{I}(P_i),\mathcal{I}(P_j))\right]^{2H}=0, \end{alignat*}
which means $(\mathcal{I}(P_1),\cdots,\mathcal{I}(P_n),(c_1,\cdots,c_n))$ is an  $H$-critical configuration of $\mathbb{H}^{d+1}$.

	However we have seen that all the geodesics in $\mathbb{H}^{d+1}$ between the points of $\mathcal{I}(\mathbb{H}^d)$ are included in $\mathcal{I}(\mathbb{H}^d)$, therefore it is clear that 
		$$  \forall ~ i \in \{1,\cdots,n\}, ~ \dim T_{\mathcal{I}(P_i)\rightarrow \{\mathcal{I}(P_j), j \neq i\} } \leq \dim \mathcal{I}(\mathbb{H}^d) < \dim\mathbb{H}^{d+1}. $$
		
		Condition \eqref{G} of Theorem \ref{thmmainresult} is not verified although there exists an $H$-fractional Lévy Brownian field indexed by $\mathbb{H}^{d+1}$. We have reached a contradiction.

\item Let us consider distinct $P_1,\cdots, P_n \in \mathbb{H}^d$ and  $c_1,\cdots,c_n \in \mathbb{R}^*$ without further assumptions. Define $c_{n+1}=-\sum_{i=1}^n c_i.$ Using $X^H_O=0$ almost surely we can write
$$\sum_{i=1}^n c_i X^H_{P_i}\\
= \sum_{i=1}^n c_i X^H_{P_i}+c_{n+1} X^H_O $$
which is not equal to zero almost surely using point 1. of the theorem. We have shown that the Gaussian vector $\left(P_1,\cdots,P_n\right)$ is nondegenerate.
\end{enumerate}
\end{proof}
\begin{Rem} It is possible to follow the same argument to provide a short proof for the same facts about $\mathbb{R}^d$, which were already known (see for example \cite{cohenistas}). A similar technique is used in \cite{bachoc2017gaussian} to show the same result on the Wasserstein space.
	\end{Rem}
	\clearpage
\appendix
\section{Appendix: technical Lemmas}
\subsection{Geodesics}
We give the following two lemmas in order to prove Theorem \ref{thmmainresult}.
 
 \begin{Lem} \label{lemtek1} Let $M$ be a complete Riemannian manifold, $T>0$, $A,B \in M$, and $g_m : [0,T] \rightarrow M$ a sequence of minimal geodesics in $M$ such that $(g_m(0))_m$ and $(g_m(T))_m$ converge in $M$.
 	
 	There exists a minimal geodesic $$g:[0,T]\rightarrow M $$ with $g(0)=\lim\limits_{m\rightarrow +\infty} g_m(0)$ and $g(T)=\lim\limits_{m\rightarrow +\infty} g_m(T)$, and a subsequence $g_{\varphi (m)}$ of $g_m$ such that $g_{\varphi (m)}$ converges uniformly towards $g$. Furthermore $g_{\varphi(m)}'$ also converges towards $g'$ uniformly.
 \end{Lem}
 \begin{proof} If we take $m$ is large enough the distance between $g_m(0)$ and $A:=\lim\limits_{m\rightarrow +\infty} g_m(0)$ is short enough so there is a unique geodesic (up to reparametrisation) between those points (see e.g. \cite{gallot}). By parallel transport along it we  identify the tangent space $T_{g_m(0)}$ to $T_A$. Because $g_m:[0,T] \rightarrow M$ is a minimal geodesic, $t\mapsto||g'_m(t)||_M$ is constant. We deduce that $$\displaystyle ||g'_m(0)||_M =\frac{d_M(g_m(0,g_m(T))}{T}$$ is bounded in $m$.  Recall that parallel transport along a curve is a linear isometry (see again \cite{gallot}). The sequence $\left(g'_m(0)\right)_{m\in \mathbb{N}}$, viewed as taking values in $T_A$ is bounded and we extract $g'_{\varphi(m)}(0)$ converging to $v \in T_A$.
 	
 	Since $M$ is complete the exponential map is defined on the whole tangent bundle. For every $t$ in $[0,T]$ we set $$g(t):= \Exp_A(t  v). $$  As a linear isometry the parallel transport is $C^\infty$, and so is the exponential map (see \cite{gallot}), hence
 	$$g_{\varphi(m)}(t)= \Exp_{g_{\varphi(m)}(0)}(t  g'_{\varphi(m)}(0)) \underset{m \rightarrow \infty}{\longrightarrow} g(t),$$
 	as well as
 	$$g'_{\varphi(m)}(t)\underset{m \rightarrow \infty}{\longrightarrow}g'(t).$$
 	Because all arguments in the exponential belong to a compact set by Heine-Cantor theorem those convergences are uniform in~$t$. As a consequence
 	$$ L(g_{\varphi(m)}) \underset{m \rightarrow \infty}{\longrightarrow} L(g). $$
 	However
 	$$L(g_m)=d_M(g_m(0),g_m(T))\underset{m \rightarrow \infty}{\longrightarrow} d_M(A,B)$$ hence $g$ is a minimal geodesic.
 \end{proof}
 
 \begin{Lem} \label{lemtek2} Let $P_1,\cdots,P_n$ be distinct points in a complete Riemannian manifold $M$ and $c$ be a $C^\infty$ curve such that $c(0)=P_n$.
 	There exists a  sequence  of positive $\varepsilon_m$ converging towards zero such that for all $i\in \{1,\cdots,n-1 \}$, there exists a minimal geodesic $$g_i : [0,d_M(P_i,P_n)] \rightarrow M $$ with $g_i(0)=P_i$ and $g_i(d_M(P_i,P_n))=P_n$,
 	$$ d_M(P_i, c(\varepsilon_m))= d_M (P_i, P_n) + \left\langle c'(0), g_i'(d_M(P_i,P_n))\right\rangle_M \varepsilon_m + O \left( \varepsilon_m^2\right).$$
 \end{Lem}
 \begin{proof} Let us set $T=d_M(P_1,P_n)$ and choose a decreasing sequence $\varepsilon_m >0 $ converging towards zero. Because $M$ is complete there exists a sequence of minimal geodesics $g_{1,m}:[0,T]\rightarrow M$ between $P_1$ and $c(\varepsilon_m)$ (see \cite{gallot}). Using Lemma \ref{lemtek1} we can assume that $g_{1,m}$ converges uniformly towards $g_1$ a minimal geodesic between $P_1$ and $c(0)=P_n$. Lemma \ref{lemtek1} also gives the convergence of $\left(g'_{1,m}(0)\right)_m$ hence it is possible to find a $C^\infty$ map $$v :~ ]-\varepsilon_0,\varepsilon_0[ ~ \rightarrow T_{P_1}M$$ such that $v(\varepsilon_m)=g'_{1,m}(0)$ for every $m\geq 0.$ We now set
 	\begin{alignat*}{1}
 	V : [0,T] &\times]-\varepsilon_0,\varepsilon_0[  \rightarrow M  \\
 	(& s, \varepsilon)   \longmapsto \Exp_{P_1}(v(\varepsilon) s),
 	\end{alignat*}		
 	If we denote by $\mathcal{L}(\varepsilon)$ the length of the curve $s\mapsto V(s,\varepsilon)$, the first variation formula (see \cite{gallot}) gives
 	$$ \frac{d}{ d\varepsilon}\mathcal{L}(\varepsilon)\Big{|}_{\varepsilon=0} \! \! =\left[\left\langle \frac{\partial}{\partial \varepsilon} V(s,\varepsilon)\Big{|}_{(s,0)}, g_1'(s)\right\rangle_{ \! \! \! M}\right]_{s=0}^{s=T} \! \! \! \! -\int_0^T \! \! \! \left\langle \frac{\partial}{\partial \varepsilon} V(s,\varepsilon)\Big{|}_{(s,0)}, \frac{D}{ds} g_1'(s)\right\rangle_{ \! \! M} \! \! \! ds. $$
 	
 	Because $g_1$ is a geodesic
 	$$\frac{D}{ds} g_1'(s)=0,$$ 
 	hence the integral term is zero. Furthermore
 	\begin{itemize}
 		\item	$ \displaystyle \frac{\partial}{\partial \varepsilon} V(s,\varepsilon)\Big{|}_{(T,0)}=c'(0)$ since $V$ is $C^\infty$ and $$V(T,\varepsilon_m)=c(\varepsilon_m) \text{ with } \lim\limits_{m \rightarrow \infty}\varepsilon_m = 0.$$
 		\item $ \displaystyle \frac{\partial}{\partial \varepsilon} V(s,\varepsilon)\Big{|}_{(0,0)}=0$ because $V(0,\varepsilon)=P_1.$
 	\end{itemize}
 	We obtain $$\frac{d }{ d \varepsilon} \mathcal{L}(\varepsilon)\Big{|}_{\varepsilon=0}=\left[\left\langle \frac{\partial}{\partial \varepsilon} V(s,\varepsilon)\Big{|}_{(s,0)}, g_1'(s)\right\rangle_{\! \! \! M}\right]_{s=0}^{s=T}=\left\langle c'(0), g_1'(T)\right\rangle_M,$$
 	
 	hence we have
 	$$\mathcal{L}(\varepsilon)=\mathcal{L}(0)+ \left\langle c'(0), g_1'(T)\right\rangle_M \varepsilon + O\left(\varepsilon ^2\right).$$ Because $s \mapsto V(s,\varepsilon_m)=g_{1,m}(s)$ and $s \mapsto V(s,0)=g_1(s)$ are minimal geodesics we get $$d_M(P_1,c(\varepsilon_m))=d_M(P_1,P_n)+ \left\langle c'(0), g_1'(T)\right\rangle_M \varepsilon_m +O\left(\varepsilon_m ^2\right).$$
 	
 	From the sequence $\varepsilon_m$ we can extract $\varepsilon_{\varphi(m)}$ using Lemma \ref{lemtek1} again and iterate the argument to get the result for every $d_M(P_i,c(\varepsilon_m))$.
 \end{proof}

\subsection{Loops}

The following lemma are useful to prove Theorem \ref{thmlacetminimal} and \ref{thmcompact}. See Section \ref{subsec:loopresults} for notations.
\begin{Lem} \label{lemlacetminimal} Let $M$ be a Riemannian manifold, and $\gamma$ a loop of minimal length in $\mathscr{L}(M)\setminus C$. Then for all $P,Q \in \gamma$, all the minimal geodesics from $P$ to $Q$ are included in $\gamma$.
\end{Lem}
\begin{proof} \begin{enumerate} \item Let us assume that $\gamma$ is of minimal length in $\mathscr{L}(M)\setminus C$ and show that every proportional to arc-length parametrisation of $\gamma$ is a geodesic. Suppose that it is not the case.  Clearly, there exists a proportional to arc-length parametrisation of $\gamma$ and $t_1 < t_2$ with $t_2-t_1$ as small as wanted such that $\gamma_{|[t_1,t_2]}$ is not a geodesic.  Now we can take $t_2-t_1$ small enough so that there exists a unique minimal geodesic $\gamma(t_1)\gamma(t_2):[0,1] \rightarrow M$ between $\gamma(t_1)$ and $\gamma(t_2)$ (see \cite{gallot}). It is clear that $\gamma(t_1)\gamma(t_2)$ is shorter than $\gamma_{|[t_1,t_2]}$ (otherwise $\gamma_{|[t_1,t_2]}$ is a minimal geodesic, which is impossible since it is not a geodesic). 
		Hence the concatenation of $\gamma(t_1)\gamma(t_2)$ with $\gamma \setminus \gamma_{|[t_1,t_2]}$ is a loop $\tilde{\gamma}$ with shorter length than $\gamma$. Now we can take $t_2-t_1$ small enough to have $\gamma_{|[t_1,t_2]}$ and $\gamma(t_1)\gamma(t_2)$ taking values in a common geodesically convex ball. Therefore $\gamma(t_1)\gamma(t_2)$ is homotopic to $\gamma_{|[t_1,t_2]}$, and finally $\gamma$ and $\tilde{\gamma}$ are homotopic. In the end $\tilde{\gamma} \in \mathscr{L}(M)\setminus C$ is shorter than $\gamma$. We have reached a contradiction.
		
		\item Now let us assume there exist $P,Q \in \gamma$ and a minimal geodesic $g$ between $P$ and $Q$, not included in $\gamma$. Without loss of generality let us assume
		$$\gamma: [0,1] \rightarrow M$$
		with
		$$\gamma(0)=\gamma (1)= Q, \gamma(t_P)=P,$$
		and $$ g: [0,1] \rightarrow M$$ with $$g(0)=P,~ g(1)=Q.$$ We define \begin{center}
			$\begin{array}{rlrl}
			\gamma_1 : [0,t_P]  & \rightarrow M &~~~~~~~~~~\gamma_2 : [t_P,1] & \rightarrow M \\
			t & \mapsto \gamma(t) & t & \mapsto  \gamma(t)
			\end{array}$\end{center} the two halves of $\gamma$ connecting $P$ to $Q$. Let us consider $l_1=\gamma_1 \cdot g$ (the concatenation of $\gamma_1$ and $g$), and $l_2= \overleftarrow{g}\cdot \gamma_2$, where $\overleftarrow{g}:t\mapsto g(1-t)$. Because $g$ is a minimal geodesic between $P$ and $Q$ we have $$L(g)\leq L(\gamma_1) \text{ and } L(g) \leq L(\gamma_2),$$ hence \begin{equation} \label{minim} L(l_1)\leq L(\gamma) \text{ and } L(l_2)\leq L(\gamma).\end{equation}
		Now $l_1$ is not $C^\infty$ at $P$: indeed $g'(0)$ cannot be equal to $\gamma'(t_P)$, otherwise we would have
		$$g(t)=\Exp_p \left(\gamma'(t_P)t\right)=\left\{\begin{array}{c}\gamma(t_P+t) \text{ if } t_P+t, \leq 1 \\ \gamma(t_P+t-1) \text{ elsewise.} \end{array}\right.$$
		(Recall a geodesic is completely determined by some initial conditions on position and speed, see e.g. \cite{gallot}). This is impossible since we assumed $g$ is not included in $\gamma$. Since $l_1$ is not $C^\infty$ it cannot be a geodesic, hence $l_1$ is not of minimal length in $\mathscr{L}(M)\setminus C$ (see part $1$ of the proof). Because of \eqref{minim}, it is then clear that $l_1$ belongs to $C$. The same is true for $l_2$. We will now show that $\gamma \in C$ to get a contradiction.
		
		Let us consider $\tilde{\gamma}_1,\tilde{\gamma}_2,\tilde{g},\overleftarrow{\tilde{g}}$ reparametrisations of $\gamma_1,\gamma_2,g,\overleftarrow{g}$ over $[0,1/2]$. It is clear that $\tilde{l}_1:=\tilde{\gamma}_1\cdot \tilde{g}$ and $\tilde{l}_2:=\overleftarrow{\tilde{g}} \cdot \tilde{\gamma}_2$ are reparametrisations over $[0,1]$ of $l_1$ and $l_2$. In particular $\tilde{l}_1, \tilde{l}_2 \in C$ are homotopic to the constant loop $\overline{Q}$.
		
		Let us consider $(s,t)\mapsto l_t^1(s)$ and $(s,t)\mapsto l_t^2(s)$ two free homotopies of loops such that $$l_0^1=\tilde{l}_1, ~ l_1^1=\overline{Q}$$ and $$l_0^2=\tilde{l}_2, ~l_1^2=\overline{Q},$$
		
		Let us now give an explicit free homotopy from $\gamma$ to $\overline{Q}$. For every $0 \leq t \leq 1/2$ let us consider
		\begin{center}
			$\begin{array}{rlrl}
			\tilde{g}_t : [0,1/2]  & \rightarrow M &~~~~~~~~~~\overleftarrow{\tilde{g_t}} : [0,1/2] & \rightarrow M \\
			s & \mapsto \tilde{g}(st) & s & \mapsto  \tilde{g}(t-st).
			\end{array}$
		\end{center}
		
		and define
		$$\gamma_t= \left\{ \begin{array}{ll} \tilde{\gamma_1} \cdot \tilde{g}_t \cdot \overleftarrow{\tilde{g_t}} \cdot \tilde{\gamma}_2, & 0\leq t \leq 1/2, \\
		l^1_{2t-1} \cdot l^2_{2t-1}, & 1/2 \leq t \leq 1. \end{array} \right.$$

		It is clear that for every $t\in [0,1]$, \begin{alignat*}{1} \gamma_t :[0,2] &\rightarrow M \\ s &\mapsto \gamma_t(s)\end{alignat*} is a loop. We reparametrise $\gamma_t$ over $[0,1]$ by setting $\tilde{\gamma}_t(s)=\gamma_t(s/2)$ to obtain a free homotopy of loops such that $\gamma_1=\overline{Q}$ and $\gamma_0$ is a reparametrisation of $\gamma$. In the end $\gamma$ is freely homotopic to $\overline{Q}$, which proves $\gamma \in C$: we have reached a contradiction. \qedhere
	\end{enumerate}
\end{proof}

\begin{Lem} \label{lemcompact}
	Let $M$ be a compact, nonsimply connected Riemannian manifold. There exists a loop $\gamma$ of minimal length in $\mathscr{L}(M)\setminus C$.
\end{Lem}
\begin{proof}Let us consider
	$$d:=\inf\{L(l), l \in  \mathscr{L}(M)\setminus C\}.$$
	Since $M$ is not simply connected $\mathscr{L}(M) \setminus C$ is not empty, hence $d>0$ and there exists a sequence $(l_n)_{n\geq 0}$ of $\mathscr{L}(M)\setminus C$ such that
	
	\begin{equation} \label{conv}L(l_n) \underset{n \rightarrow \infty}{\longrightarrow} d. \end{equation}
	Let us reparametrise $l_n$ over $[0,1]$ and proportionally to arc-length in order to have
	
	$\forall t_1\leq t_2 \in [0,1]$,
	$$ d_M (l_n (t_1), l_n (t_2)) \leq \int_{t_1}^{t_2} ||l'_n(t)||_Mdt \leq \sup(L(l_n))(t_2-t_1),$$
	which shows that the set $\{ l_n \}$ is equicontinuous. Because $M$ is compact, for all $t \in[0,1]$, $\{l_n(t) , n \in \mathbb{N}\}$ is relatively compact. By Arzelà-Ascoli theorem we conclude that $\{ l_n\}$ is relatively compact in the uniform topology, hence we can extract a subsequence of $l_n$ which converges uniformly to $l_\infty$ a continuous loop in $M$.
	To obtain a piecewise continuously differentiable curve, we consider a partition of $[0,1]$ $$0=t_0 < t_1 < \cdots < t_m=1$$ with successive times $t_k,t_{k+1}$ close enough to each other so that every $l_{\infty|[t_k,t_{k+1}]}$ takes value in a geodesically convex ball of $M$. We now consider the closed curve $\gamma$ given by the concatenation of all the minimal geodesics $l_\infty(t_k)l_\infty(t_{k+1})$.
	
	Let us notice that for $n$ large enough $l_{n|[t_k,t_{k+1}]}$ takes values in the same geodesically convex ball as $\gamma_{|[t_k,t_{k+1}]}=l_\infty(t_k)l_\infty(t_{k+1})$, hence $l_{n|[t_k,t_{k+1}]}$ and $\gamma_{|[t_k,t_{k+1}]}$ are homotopic. In the end $l_n$ is homotopic to $\gamma$, which shows that $\gamma \in \mathscr{L}(M)\setminus C$, hence $L(\gamma) \geq d$. We will now show that $L(\gamma) = d$ to finish the proof.
	
	Let us assume that $L(\gamma) > d$. We write
	\begin{alignat*}{1}
	L(\gamma) & =L(\gamma)-d+d-L(l_n)+L(l_n).
	\end{alignat*}
	We now use \eqref{conv} to write $d-L(l_n)>-\varepsilon$ for $n$ large enough. We obtain:
	\begin{alignat*}{1}
	L(\gamma) & > L(\gamma) - d - \varepsilon + L(l_n) \\
	\Longleftrightarrow \sum_{k=0}^{m-1} L(\gamma_{|[t_k,t_{k+1}]}) & > \sum_{k=0}^{m-1}\left(\frac{L(\gamma) - d -\varepsilon }{m} + L(l_{n|[t_k,t_{k+1}]})\right),
	\end{alignat*}
	which ensures there exists $k$ such that
	\begin{equation} \label{fineq}
	L(\gamma_{|[t_k,t_{k+1}]}) > \frac{L(\gamma) - d -\varepsilon }{m} + L(l_{n|[t_k,t_{k+1}]}).\\
	\end{equation}
	Now let us recall that $$l_n(t_k)\underset{n \rightarrow \infty}{\longrightarrow} \gamma(t_k),$$  $$l_n(t_{k+1})\underset{n \rightarrow \infty}{\longrightarrow} \gamma(t_{k+1}),$$
	hence for $n$ large enough $$\frac{L(\gamma) - d -\varepsilon }{m}> d_M(l_n(t_k),\gamma(t_k))+d_M(l_n(t_{k+1}),\gamma(t_{k+1})).$$
	
	Together with \eqref{fineq} we obtain
	
	$$ L(\gamma_{|[t_k,t_{k+1}]}) > d_M(l_n(t_k),\gamma(t_k))+d_M(l_n(t_{k+1}),\gamma(t_{k+1}))  + L(l_{n|[t_k,t_{k+1}]}),$$ 
	
	which contradicts the fact that $\gamma_{|[t_k,t_{k+1}]}=\gamma(t_k)\gamma(t_{k+1})$ is a minimal geodesic.
\end{proof}
 
 \section*{Acknowledgements} The author is grateful to Serge Cohen who carefully advised him during this work.
 
\bibliographystyle{plain}
\bibliography{Nil}
 
\end{document}